\documentclass{amsart}

\usepackage{amsfonts,amssymb,amsmath,enumerate,verbatim,mathtools,tikz,bm,mathrsfs,hyperref,comment,stmaryrd,amsthm}
\usepackage{xcolor}
\usetikzlibrary{decorations.pathreplacing}

\usepackage{colonequals}
\usepackage{adjustbox}
\usepackage{leftindex}
\usepackage[all,pdf]{xy}
\usepackage{tikz-cd}
\usepackage{enumitem}
\hypersetup{
    colorlinks=true,
    linkcolor=blue,
    filecolor=magenta,      
    urlcolor=cyan,
}

\usepackage{thmtools}
\usepackage{cleveref}

\DeclareMathOperator{\op}{\mathsf{op}}
\DeclareMathOperator{\Ima}{Im}

\DeclareMathOperator{\Coker}{Coker}

\DeclareMathOperator{\Tr}{Tr}
\DeclareMathOperator{\add}{add}

\DeclareMathOperator{\Hom}{Hom}
\DeclareMathOperator{\Ext}{Ext}

\newcommand{\del}{\partial}
\newcommand{\Gproj}{\mathrm{GP}}
\newcommand{\mo}{\mbox{-}\mathrm{mod}}
\newcommand{\proj}{\mathcal{P}}

\newcommand{\perpone}[1]{{}^{\perp_1}#1}

\newtheorem{theorem}{Theorem}[section]

\newtheorem{proposition}[theorem]{Proposition}

\newtheorem{lemma}[theorem]{Lemma}
\newtheorem{corollary}[theorem]{Corollary}

\theoremstyle{definition}

\newtheorem{chunk}[theorem]{}

\usepackage[utf8]{inputenc}
\begin{document}

\title[A note on the $\mho$-quiver]{A note on the $\mho$-quiver }
\author[Xue-Song Lu] {Xue-Song Lu}
\address{School of Mathematical Sciences, Shanghai Jiao Tong University, Shanghai 200240, P.R. China, https://orcid.org/0009-0001-4439-4684}
\email{leocedar@sjtu.edu.cn}
\thanks{Supported by the National Key Research and Development Project 2025YFA1017202, and National Natural Science Foundation of China with Grant No. 12131015.}

\begin{abstract}
A two-parameter hierarchy of classes of finitely generated modules over an artin algebra is described by using the $\mho$-quiver. As an application, it is shown that the stable category of reflexive modules is equivalent to two other categories via the mutually quasi-inverse equivalences induced by $\Omega$ and $\mho$. 

\vskip5pt


\end{abstract}

\maketitle

\section{Introduction}
The $\mho$-quiver of an artin algebra was introduced by C. M. Ringel and P. Zhang in
\cite{RZ}. It serves as a useful tool to study the homological properties of modules. In particular, Gorenstein-projective, semi-Gorenstein-projective and $t$-torsionfree modules admit natural interpretations in the $\mho$-quiver. 

\vskip5pt

The aim of this note is to give a two-parameter hierarchy of classes of finitely generated modules over an artin algebra through syzygies of the injective cogenerator and cosyzygies of the projective generator, by using the $\mho$-quiver. For details, see \Cref{quiverthm}. 

\vskip5pt

Reflexive modules play an important role in several areas. For example, in representation theory of quotient singularities, the Auslander--Reiten quiver of reflexive modules realizes the McKay quiver of the corresponding finite group (\cite{A}); in noncommutative birational geometry, noncommutative crepant resolutions are constructed as endomorphism algebras of suitable reflexive modules (\cite{VdB}); in algebraic geometry, reflexive sheaves provide natural substitutes for vector bundles on singular varieties (\cite{Har}). In this note, as an application of the main result, we show that the stable category of finitely generated reflexive modules of an artin algebra is equivalent to two other categories via the mutually quasi-inverse equivalences induced by $\Omega$ and $\mho$. For details, see \Cref{reflexiveequivalence}. 

\vskip5pt

The notion of a $t$-torsionfree module was introduced by M. Auslander and
M. Bridger in \cite{AB} as a higher analogue of torsionless and reflexive modules. Indeed, $1$-torsionfree modules are precisely the torsionless modules, while $2$-torsionfree modules are precisely the reflexive modules. The notion has many applications both in representation theory of artin algebras and commutative noetherian rings; for example, see \cite{HH}, \cite{MTT}, \cite{DT}, \cite{O} and \cite{CM}. For an artin algebra $\Lambda$ and $t\geq 1$, a $\Lambda$-module $M$ is $t$-torsionfree (respectively, $\infty$-torsionfree) if $\Ext_{\Lambda^{\op}}^i(\Tr M, \Lambda)=0$ for $1\leq i\leq t$ (respectively, for $i\geq 1$), where $\Tr$ is the Auslander transpose. In this note, as a byproduct, we show that $t$-torsionfree modules also have strong relations with $\Ext_\Lambda^1(-, \mathcal S_t)$ for some class of $\Lambda$-modules $\mathcal S_t$. For details, see \Cref{torsionfree} and \Cref{torsionfreeinverse}. 

\section{The \texorpdfstring{$\mho$}{agemo}-quiver of an artin algebra}
Throughout this note, $\Lambda$ is always an artin algebra over a commutative artin ring $R$. A module always means a finitely generated left $\Lambda$-module, whose category is denoted by $\Lambda\mo$. Denote by $D$ the canonical duality functor between $\Lambda\mo$ and $\Lambda^{\op}\mo$. Denote by $\proj$ the full subcategory of projective modules. 

\vskip5pt

Let $M\in \Lambda\mo$. Then $M$ has a projective cover $P\longrightarrow M$, whose kernel is denoted by $\Omega M$, called the first syzygy of $M$. $M$ also has a minimal left $\proj$-approximation $f: M\longrightarrow Q$. That is, $Q\in \proj$ and $f$ satisfies that any morphism $g: M\longrightarrow Q'$ with $Q'\in \proj$ factors through $f$ and if $f=tf$ for $t: Q\longrightarrow Q$, then $t$ must be an automorphism. The existence of such an $f$ is well-known; see \cite[Proposition 4.2]{AS} and \cite[Theorem 2.2]{ARS}.  

\vskip5pt

Moreover, set $\mho M\colonequals\Coker f$. Following \cite{RZ}, $\mho$ is pronounced ``agemo'', reminding that it is a kind of inverse of $\Omega$, as shown in the following proposition. Recall that a module is called torsionless if it is a submodule of some projective module. A module is torsionless if and only if its minimal left $\proj$-approximation is monic. 

\begin{proposition}\label{agemo}\cite[Lemma 3.2]{RZ}
Let $M\in \Lambda\mo$. If $M$ is indecomposable, non-projective and torsionless, then there is a short exact sequence
\[
\xymatrix{
0\ar[r]& M\ar[r]^-f& Q \ar[r] & \mho M\ar[r]& 0
}
\]
where $f$ is a minimal left $\proj$-approximation and $\mho M$ is indecomposable non-projective. In particular, $\Omega\mho M\cong M$. 
\end{proposition}

The $\mho$-quiver of $\Lambda$ has as vertices the isomorphism classes $[M]$ of indecomposable non-projective modules, and there is an arrow 
\[
\xymatrix{
[M]&&[N]\ar@{-->}[ll]
}
\]
whenever $M$ is torsionless and $N=\mho M$. 

\vskip5pt

A basic property of an $\mho$-quiver is that, for any vertex $[M]$, there is at most one arrow starting at $[M]$ and at most one arrow ending at $[M]$; see \cite[Page 4]{RZ}. Thus, a component in the $\mho$-quiver, which is called an $\mho$-component, must be one of the following five types: a linearly oriented quiver $\mathbb{A}_n$, an oriented cycle $\tilde{\mathbb{A}}_n$, $-\mathbb{N}$, $\mathbb{N}$ or $\mathbb{Z}$. 
\[
\begin{tikzcd}[column sep=1.6em, row sep=0.6em]
\cdots
& \circ \arrow[l, dashed]
& \circ \arrow[l, dashed]
& \circ \arrow[l, dashed]
&&
\circ
& \circ \arrow[l, dashed]
& \circ \arrow[l, dashed]
& \cdots \arrow[l, dashed]
\\
& -\mathbb{N}
&&&&
&& \mathbb{N}
\end{tikzcd}
\]

\vskip5pt 

Many homological properties of an indecomposable non-projective module can be read from its position in the component of the $\mho$-quiver. A module $M$ is called semi-Gorenstein-projective if $\Ext_\Lambda^i(M, \Lambda)=0$ for all $i>0$; $M$ is called $t$-torsionfree (respectively, $\infty$-torsionfree) if $\Ext_{\Lambda^{\op}}^i(\Tr M, \Lambda)=0$ for $1\leq i\leq t$ (respectively, for $i\geq 1$); $M$ is called Gorenstein-projective if it is both semi-Gorenstein-projective and $\infty$-torsionfree. A module is $1$-torsionfree if and only if it is torsionless, and a $2$-torsionfree module is called reflexive. 

\begin{theorem}\label{RZ} \cite[Theorem 1.5, Remark 3]{RZ}  Let $\Lambda$ be an artin algebra, and $M$ be an indecomposable non-projective $\Lambda$-module.
    \vskip5pt
\begin{enumerate}
\item[{\rm(0)}] \ $[M]$ is an isolated vertex if and only if $\Ext_\Lambda^1(M,\Lambda) \neq 0$ and $M$ is not torsionless.
\vskip5pt
\item[{\rm(1)}] \ There is a path of length $t\geq 1$ starting at $[M]$ if and only if $\Ext_\Lambda^i(M,\Lambda) = 0$ for $1\le i \le t$. In particular, $[M]$ is the start of an arrow if and only if $\Ext_\Lambda^1(M,\Lambda) = 0$.
\vskip5pt
\item[{\rm(1$'$)}] \ There is a path of length $\infty$ starting at $[M]$ if and only if $M$ is semi-Gorenstein-projective.
\vskip5pt
\item[{\rm(1$''$)}] \ The $\mho$-component of $[M]$ is $-\mathbb N$ if and only if $M$ is semi-Gorenstein-projective, but not Gorenstein-projective. 
\vskip5pt
\item[{\rm(2)}] \ There is a path of length $t\geq 1$ ending at $[M]$ if and only if $M$ is $t$-torsionfree. In particular,  $[M]$ is the end of an arrow if and only if $M$ is torsionless.
\vskip5pt
\item[{\rm(2$'$)}] \ There is a path of length $\infty$ ending at $[M]$ if and only if $M$ is $\infty$-torsionfree. 
\vskip5pt
\item[{\rm(2$''$)}] \ The $\mho$-component of $[M]$ is $\mathbb N$ if and only if $M$ is $\infty$-torsionfree, but not Gorenstein-projective.
\vskip5pt
\item[{\rm(3)}] \  There is both a path of length $\infty$ starting at $[M]$ and a path of length $\infty$ ending at $[M]$ if and only if $M$ is Gorenstein-projective.
\end{enumerate}
\end{theorem}

\section{A hierarchy of subcategories of modules}
 We begin with some notation. Choose a minimal projective resolution of $D(\Lambda_\Lambda)$
\[
\xymatrix@R=0.2cm{
\cdots\ar[r]& P_2\ar[rr]^-{\del_2}\ar@{->>}[dr] & & P_1\ar@{->>}[dr]\ar[rr]^-{\del_1}& & P_0\ar[r]^-\pi &D(\Lambda_\Lambda)=K_0\ar[r]&0,\\
&&K_2\ar@{^(->}[ur]&&K_1\ar@{^(->}[ur]&&&
}
\]
where $K_i$ is the image of $\del_i$ for $i\geq 1$.

\vskip5pt

Also, choose a minimal injective coresolution of $\leftindex_{\Lambda}{\Lambda}$
\[
\xymatrix@R=0.2cm{
0\ar[r] & \leftindex_{\Lambda}\Lambda=L_0\ar[r] & I_0\ar[rr]^-{\del_{-1}}\ar@{->>}[dr] & & I_{-1}\ar[rr]^-{\del_{-2}}\ar@{->>}[dr] & & I_{-2}\ar[r]^-{\del_{-3}}&\cdots, \\
& & & L_{-1}\ar@{^(->}[ur] & & L_{-2}\ar@{^(->}[ur] & &
}
\]
where $L_i$ is the image of $\del_i$ for $i\leq -1$.

\vskip5pt

In the rest of this paper, for $n\in \mathbb N$, we define
\[
\mathcal S_n\colonequals \{K_0, K_1, \cdots, K_n\},\qquad
\mathcal S=\mathcal S_\infty\colonequals\bigcup\limits_{n\in \mathbb N} \mathcal S_n
=\{K_0, K_1, \cdots\}, 
\]
and
\[
\mathcal T_n\colonequals\{L_0, L_{-1}, \cdots, L_{-n}\},\qquad
\mathcal T=\mathcal T_\infty\colonequals \bigcup\limits_{n\in \mathbb N} \mathcal T_n
= \{L_0, L_{-1}, \cdots\}.
\]
For $\mathcal X\subseteq \Lambda\mo$, let $\perpone{\mathcal X}$ denote the full subcategory
\[
\{Z\in\Lambda\mo\mid \Ext^1_\Lambda(Z,X)=0 \text{ for all }X\in\mathcal X\}.
\]

\begin{lemma}\label{extensionless}
Let $M$ be a non-projective indecomposable $\Lambda$-module and $t\in \mathbb Z^+\cup \{\infty\}$. Then there is a path of length $t$ starting at $[M]$ if and only if $M\in \perpone{\mathcal T_{t-1}}$. 
\end{lemma}
\begin{proof}
    It follows directly from the definitions and \Cref{RZ}(1)(1'). 
\end{proof}

\begin{proposition}\label{torsionfree}
 Let $M$ be a $\Lambda$-module and $t\in \mathbb Z^+\cup \{\infty\}$. If $M\in \perpone{\mathcal S_t}$, then $M$ is $t$-torsionfree. 
\end{proposition}
\begin{proof}
It suffices to consider the case $t< \infty$ and $M$ is indecomposable non-projective. We give a proof by induction on $t$. We first show that $M$ is torsionless if $M\in \perpone{\mathcal S_1}$. Choose an embedding $h\colon M \rightarrow D(\Lambda_\Lambda)^n$. Since $\Ext_\Lambda^1(M, K_1^n)=0$, $h$ factors through $\pi^n\colon P_0^n\rightarrow D(\Lambda_\Lambda)^n$, say with $h=\pi^n g$, where $g\colon M\rightarrow P_0^n$. $g$ is a monomorphism since $h$ is a monomorphism. Thus, $M$ is torsionless. 
\[
\xymatrix{
&&&M\ar[d]^-{h}\ar@{..>}[ld]_-{g}&\\
0\ar[r]& K_1^n\ar[r]& P_0^n \ar[r]^-{\pi^n} & D(\Lambda_\Lambda)^n\ar[r]& 0
}
\]
Now suppose that $M\in \perpone{\mathcal S_t}$, $t\geq 2$. By \Cref{RZ}(2) and induction, it suffices to prove that $\mho M\in \perpone{\mathcal S_{t-1}}$. For $1\leq i\leq t-1$, the two exact sequences $ 0\longrightarrow K_{i+1}\longrightarrow P_{i}\longrightarrow K_{i}\longrightarrow 0$ and $0\longrightarrow M\longrightarrow Q\longrightarrow \mho M\longrightarrow 0$ yield a commutative diagram
 \[
    \xymatrix{
    & 0\ar[d] & 0 \ar[d]& 0\ar[d]& \\
    0\ar[r] & \Hom_{\Lambda}(\mho M, K_{i+1})\ar[r]\ar[d]& \Hom_{\Lambda}(\mho M, P_i)\ar[r]\ar[d]& \Hom_{\Lambda}(\mho M, K_i)\ar[d]& \\
    0\ar[r] & \Hom_{\Lambda}(Q, K_{i+1})\ar[r]\ar[d]& \Hom_{\Lambda}(Q, P_i)\ar[r]\ar[d]& \Hom_{\Lambda}(Q, K_i)\ar[d]\ar[r]&0 \\
    0\ar[r] & \Hom_{\Lambda}(M,K_{i+1})\ar[r]& \Hom_{\Lambda}(M, P_i)\ar[r]\ar[d]& \Hom_{\Lambda}(M, K_i)\ar[r]\ar[d]& 0\\
    &&0&\Ext_\Lambda^1(\mho M, K_i)\ar[d]&\\
    &&&0&
    }
    \]
    which shows that $\Ext_\Lambda^1(\mho M, K_i)=0$. 
\end{proof}

\Cref{torsionfree} has a partial converse. 

\begin{proposition}\label{torsionfreeinverse}
   Let $M$ be a $\Lambda$-module and $t\in \mathbb Z^+\cup \{\infty\}$. If $\Ext^1_\Lambda(M,\Lambda)=0$ and $M$ is $t$-torsionfree, then $M\in \perpone{\mathcal S_t}$. 
\end{proposition}
\begin{proof}
 It suffices to consider the case $t< \infty$ and $M$ is indecomposable non-projective. Also, one only needs to prove that $\Ext_\Lambda^1(M, K_t)=0$ since a $t$-torsionfree module is also $(t-1)$-torsionfree. By \Cref{agemo} and \Cref{RZ}(2), $M\cong \Omega^t\mho^t M$. Thus, by dimension shifting, $\Ext_\Lambda^1(M, K_t)=\Ext_\Lambda^{t+1}(\mho^t M, K_t)$. 
 
 \vskip5pt
 
 By \Cref{RZ}(1)(2), $\Ext_\Lambda^{i}(\mho^t M, \Lambda)=0$ for $1\leq i\leq t$. And $\Ext_\Lambda^{t+1}(\mho^t M, \Lambda)=\Ext_\Lambda^{1}(M, \Lambda)=0$. Then, apply $\Hom_{\Lambda}(\mho^t M,-)$ to the short exact sequences $\{ 0\longrightarrow K_{i+1}\longrightarrow P_{i}\longrightarrow K_{i}\longrightarrow 0\}$. They yield that $\Ext_\Lambda^{t+1}(\mho^t M, K_t)=\Ext_\Lambda^{t}(\mho^t M, K_{t-1})=\cdots = \Ext_\Lambda^1(\mho^t M, K_0)=\Ext_\Lambda^1(\mho^t M, D(\Lambda_\Lambda))=0$. This completes the proof. 
\end{proof}

\begin{theorem}\label{quiverthm}
    Let $M$ be an indecomposable non-projective $\Lambda$-module. Suppose that $m, n\in \mathbb N\cup \{\infty\}$ and $m\ne 0$. Then there is both a path of length $m$ starting at $[M]$ and a path of length $n$ ending at $[M]$ if and only
if
\[
  M\in\perpone{(\mathcal T_{m-1}\cup\mathcal S_n)}.
\]
In particular, there is both an arrow starting at $[M]$ and an arrow ending at
$[M]$ if and only if
\[
  M\in\perpone{(\Lambda\oplus K_1)}.
\]
\end{theorem}
\begin{proof}
    The case $n=0$ follows from \Cref{extensionless}, since $\mathcal S_0=\{D(\Lambda_\Lambda)\}$. Now suppose that $n\geq 1$. Note that $\perpone{(\mathcal T_{m-1}\cup\mathcal S_n)}=\perpone{\mathcal T_{m-1}}\cap \perpone{\mathcal S_n}$. If there is both a path of length $m$ starting at $[M]$ and a path of length $n$ ending at $[M]$, then by \Cref{extensionless}, \Cref{RZ}(2)(2') and \Cref{torsionfreeinverse}, we have $ M\in\perpone{\mathcal T_{m-1}}\cap \perpone{\mathcal S_n}$. Conversely, if $M \in\perpone{\mathcal T_{m-1}}\cap \perpone{\mathcal S_n}$, then both paths exist by \Cref{extensionless}, \Cref{RZ}(2)(2') and \Cref{torsionfree}. 
\end{proof}

We draw a diagram to illustrate \Cref{quiverthm}. Let $C_{m,n}$ be the class of indecomposable non-projective $\Lambda$-modules $M$ such that there is both a path of length $m$ starting at $[M]$ and a path of length $n$ ending at $[M]$, and $\mathfrak C_{m,n}\coloneqq \add(C_{m,n}\cup \proj)$ be the additive subcategory generated by $C_{m,n}$ and $\proj$. Then the classes $\left\{\mathfrak C_{m,n}\mid m,n\in \mathbb N\cup\{\infty\}\right\}$ form a lattice:

\vskip5pt

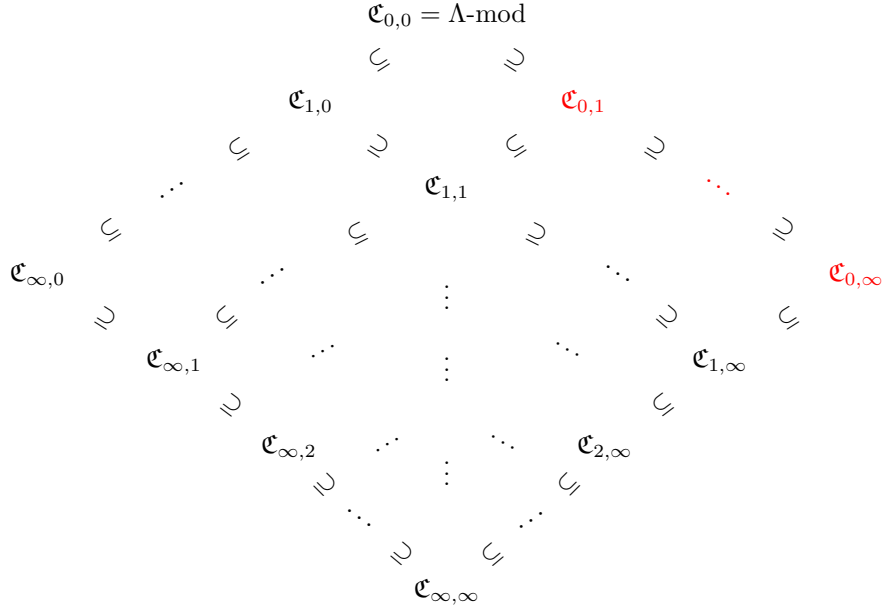
\begin{tikzpicture}[scale=.95]

\tikzset{
  every node/.style={
    inner sep=1.2pt
  },
  inclusion/.style={
    font=\small,
    inner sep=0pt,
    fill=white,
    sloped
  },
  omitted/.style={
    font=\small,
    inner sep=0pt
  }
}

\newcommand{\inclR}[2]{%
  \path (#1) -- node[midway,inclusion] {$\subseteq$} (#2);
}

\newcommand{\inclL}[2]{%
  \path (#1) -- node[midway,inclusion] {$\supseteq$} (#2);
}

\node (c00) at (0,8.0)
  {$\mathfrak C_{0,0}=\Lambda\mo$};

\node (c10) at (-1.9,6.8)
  {$\mathfrak C_{1,0}$};

\node[text=red] (c01) at (1.9,6.8)
  {$\mathfrak C_{0,1}$};

\node (c11) at (0,5.6)
  {$\mathfrak C_{1,1}$};

\node[rotate=35] (dots30) at (-3.8,5.6)
  {$\cdots$};

\node[text=red][rotate=-35] (dots03) at (3.8,5.6)
  {$\cdots$};

\node (cinf0) at (-5.7,4.4)
  {$\mathfrak C_{\infty,0}$};

\node[text=red] (c0inf) at (5.7,4.4)
  {$\mathfrak C_{0,\infty}$};

\node[rotate=35] (dots31) at (-2.4,4.4)
  {$\cdots$};

\node[rotate=-35] (dots13) at (2.4,4.4)
  {$\cdots$};

\node (cinf1) at (-3.8,3.2)
  {$\mathfrak C_{\infty,1}$};

\node (c1inf) at (3.8,3.2)
  {$\mathfrak C_{1,\infty}$};

\node (cinf2) at (-2.2,2.0)
  {$\mathfrak C_{\infty,2}$};

\node (c2inf) at (2.2,2.0)
  {$\mathfrak C_{2,\infty}$};

\node [rotate=-35](vdotL) at (-1.2,1.0)
  {$\cdots$};

\node[rotate=35] (vdotR) at (1.2,1.0)
  {$\cdots$};

\node (cinfinf) at (0,0)
  {$\mathfrak C_{\infty,\infty}$};

\node[omitted]            at ( 0.00,4.20) {$\vdots$};
\node[omitted,rotate=30]  at (-1.70,3.35) {$\cdots$};
\node[omitted]            at ( 0.00,3.20) {$\vdots$};
\node[omitted,rotate=-30] at ( 1.70,3.35) {$\cdots$};
\node[omitted,rotate=30]  at (-0.80,2.05) {$\cdots$};
\node[omitted]            at ( 0.00,1.75) {$\vdots$};
\node[omitted,rotate=-30] at ( 0.80,2.05) {$\cdots$};
\inclR{c10}{c00}
\inclL{c01}{c00}
\inclL{c11}{c10}
\inclR{c11}{c01}
\inclR{dots30}{c10}
\inclR{cinf0}{dots30}
\inclR{dots31}{c11}
\inclR{cinf1}{dots31}
\inclL{dots03}{c01}
\inclL{c0inf}{dots03}
\inclL{dots13}{c11}
\inclL{c1inf}{dots13}
\inclL{cinf1}{cinf0}
\inclL{cinf2}{cinf1}
\inclL{vdotL}{cinf2}
\inclL{cinfinf}{vdotL}
\inclR{c1inf}{c0inf}
\inclR{c2inf}{c1inf}
\inclR{vdotR}{c2inf}
\inclR{cinfinf}{vdotR}

\end{tikzpicture}

\vskip5pt

where all $\mathfrak C_{m,n}$ are the left-$\Ext^1$-orthogonal classes $\perpone{(\mathcal T_{m-1}\cup \mathcal S_{n})}$, except the red ones (define $\mathcal T_{-1}\coloneqq \varnothing $, then $\mathfrak C_{0,0}=\Lambda \mo=\perpone{\mathcal S_0}$; but in general one only has $\mathfrak C_{0,t}\supseteq \perpone{\mathcal S_t}$ for $t\in \mathbb{Z}^+\cup \{\infty\}$, as shown in \Cref{torsionfree}). 

\vskip5pt

In particular, there is an immediate corollary, which was proved directly in \cite{DLL}. 

\begin{corollary}
   Denote by $\Gproj(\Lambda)$ the class of finitely generated Gorenstein-projective $\Lambda$-modules. Then we have
$$
\Gproj(\Lambda)=\perpone{(\mathcal T\cup \mathcal S)}. 
$$
\end{corollary}
\begin{proof}
By \Cref{RZ}(3), $\Gproj(\Lambda)=\mathfrak C_{\infty,\infty}$, which is equal to $\perpone{(\mathcal T\cup \mathcal S)}$ by \Cref{quiverthm}. 
\end{proof}

The (projective) stable category of $\Lambda$ is defined as the additive quotient category $\Lambda\mbox{-}\underline{\rm mod}\coloneqq \Lambda\mo/\proj$; see \cite[IV.1]{ARS} or \cite[2.2]{K}. Following \cite[(1.1)]{RZ2}, a $\Lambda$-module $M$ satisfying $\Ext_\Lambda^1(M, \Lambda)=0$ is called extensionless. Let $\mathcal L$ be the full subcategory of torsionless $\Lambda$-modules, $\mathcal Z$ be the full subcategory of extensionless $\Lambda$-modules. By \cite[4.7(3)]{RZ}, $\mho$ and $\Omega$ induce stable equivalences: 
\[
\mathcal L/ \proj
\quad
\mathrel{
\begin{gathered}
\xrightarrow{\hspace{1cm}\mho\hspace{1cm}}
\\[-2ex]
\xleftarrow[\hspace{1cm}\Omega\hspace{1cm}]{}
\end{gathered}
}
\quad
\mathcal Z /\proj
\]

Combining this result with \Cref{extensionless}, \Cref{RZ}(1)(2) and \Cref{quiverthm}, we have an analogue for reflexive modules. 

\begin{proposition}\label{reflexiveequivalence}
Let $\mathcal R$ be the full subcategory of reflexive $\Lambda$-modules, $\mathcal Y\coloneqq\perpone{(\Lambda\oplus K_1)}$, and $\mathcal W\coloneqq\perpone{(\Lambda\oplus L_{-1})}$. Then, $\mho$ and $\Omega$ induce stable equivalences:  
\[
\mathcal R/ \proj
\quad
\mathrel{
\begin{gathered}
\xrightarrow{\hspace{1cm}\mho\hspace{1cm}}
\\[-2ex]
\xleftarrow[\hspace{1cm}\Omega\hspace{1cm}]{}
\end{gathered}
}
\quad
\mathcal Y/\proj 
\quad
\mathrel{
\begin{gathered}
\xrightarrow{\hspace{1cm}\mho\hspace{1cm}}
\\[-2ex]
\xleftarrow[\hspace{1cm}\Omega\hspace{1cm}]{}
\end{gathered}
}
\quad
\mathcal W/\proj 
\]
\end{proposition}
\begin{proof}
Note that $\mathcal Y=\mathcal L\cap \mathcal Z$. Moreover, in the stable category $\Lambda\mbox{-}\underline{\rm mod}$, $\mathcal R=\Ima (\Omega\mid_\mathcal Y)$, $\mathcal W=\Ima (\mho\mid_\mathcal Y)$, and $\mathcal Y= \Ima (\mho\mid_\mathcal R)=\Ima (\Omega\mid_\mathcal W)$. 
\end{proof}

\end{document}